\documentclass{amsart}
\usepackage{amsmath, amssymb, amsthm, amsfonts}
\usepackage{hyperref}
\usepackage{cleveref}
\usepackage[a4paper,left=20mm,right=20mm,top=30mm,bottom=25mm]{geometry}
\usepackage{blkarray}
\usepackage{enumerate}
\usepackage{xcolor}

\usepackage{graphicx}
\usepackage[all]{xy} 
\usepackage{mathrsfs}
\usepackage{eufrak}
\usepackage{comment}

\theoremstyle{plain}
\newtheorem{thm}{Theorem}[section]
\newtheorem{lem}[thm]{Lemma}
\newtheorem{prop}[thm]{Proposition}
\newtheorem{cor}[thm]{Corollary}

\theoremstyle{definition}
\newtheorem{defn}[thm]{Definition}
\newtheorem{Q}[thm]{Question}
\theoremstyle{remark}
\newtheorem{rem}[thm]{Remark}

\newcommand{\Spec}{\textrm{Spec}}

\newcommand{\Sym}{\textrm{Sym}}
\newcommand{\lra}{\longrightarrow}

\begin{document}

\title{Singularities of the nested Hilbert scheme of points of length 3,4}

\author{Doyoung Choi} 
\date{\today}

\address{Doyoung Choi \\
School of Mathematics, Korea Institute for Advanced Study (KIAS)
85 Hoegi-ro, Dongdaemun-gu, Seoul 02455,Republic of Korea}
\email{cdy4019@kias.re.kr}

\begin{abstract} 
    We show that the projection morphism $X^{[3,4]} \lra X^{[3]}$ is flat even if it has reducible fiber. After analyzing blow-up constructions related to $X^{[3,4]}$, we conclude that $X^{[3,4]}$ has canonical Gorenstein singularities. As a corollary, we specify the singularities of several nested Hilbert schemes.
\end{abstract}

\keywords{Hilbert scheme, nested Hilbert scheme, Canonical singularity, Gorenstein}
\subjclass[2020]{Primary:14B05, 14C05; Secondary:14J17}

\maketitle
\section{Introduction}

We assume that $X$ is a smooth complex projective variety of dimension $n$. Let $X^{[k]}$ denote the Hilbert scheme of $k$-points, which parameterizes length $k$ subscheme of $X$. For $n \leq 2$ or $k \leq 3$, $X^{[k]}$ is a complex projective manifold. A natural generalization is the nested Hilbert scheme. Let $\ell > 1$ and $k_1 < \cdots < k_{\ell}$ be positive integers. We denote by $X^{[k_1, \cdots, k_{\ell}]}$ the nested Hilbert scheme parameterizing nested sequence of subschemes. Set-theoretically, it is given by 
  $$ \{ (\xi_1, \cdots, \xi_{\ell}) \in X^{[k_1]} \times \cdots \times X^{[k_{\ell}]} | \xi_1 \subset \cdots \subset \xi_{\ell}  \}. $$
   The nested Hilbert scheme $X^{[k_1, \cdots, k_{\ell}]}$ forms a closed subscheme of the product $X^{[k_1]} \times \cdots \times X^{[k_{\ell}]}$. When $\ell = 2$, we refer to the corresponding scheme as a two-step nested Hilbert scheme, denoted by $X^{[k_1,k_2]}$, which parameterizes pairs of subschemes $(\xi_{k_1}, \xi_{k_2})$ such that $\xi_{k_1} \subset \xi_{k_2}$. A fundamental example is the two-step nested Hilbert scheme $X^{[1,k]}$, which coincides with the universal family $\mathcal{Z}_k$ of length $k$ subschemes of $X$. It consists of pairs $(x, \xi_k)$ where $x \in \textrm{Supp} (\xi_k)$.

  In \cite{Cheah.98}, smooth nested Hilbert scheme over $\mathbb{C}$ is classified. $X^{[k_1, \cdots,  k_{\ell}]}$ is smooth only when one of following conditions are satisfied :
\begin{enumerate}
    \item[(1)] $X$ is a curve.
    \item[(2)] $X$ is a surface, $\ell = 2$ and $k_{2} = k_1+1$.
    \item[(3)] If $\dim X \geq 3$, $\ell = 2, k_2 = k_1 + 1$ and $k_1 \leq 2$.
\end{enumerate}

 These results were established through the calculation of the dimensions of tangent spaces of nested Hilbert schemes. Although we can not directly generalize previous methods, there has been progress in understanding the geometry of singular nested Hilbert scheme of surfaces. It was shown in \cite{Ryan.Taylor.22} that the nested Hilbert scheme  $X^{[k,k+1,k+2]}$ has canonical Gorenstein singularities. 
Additionally, \cite{Ram.Sam.24} showed that $X^{[2, k]}$ has rational singularities by showing the flatness of the morphism $X^{[2,k]} \lra X^{[2]}$.

 For a smooth surface $X$, consider the universal family $\mathcal{Z}_k \cong X^{[1,k]}$ of $X^{[k]}$. It is known that $\mathcal{Z}_k$ is normal, Cohen-Macaulay(cf.\cite{Fogarty.73}) and has has non-$\mathbb{Q}$-Gorenstein rational singularities by \cite{Song.16}. 
 For higher dimensional smooth complex projective variety $X$, we consider the morphism $(\textrm{res}_{2,3}, \pi_2) : X^{[2,3]} \lra \mathcal{Z}_3$, naturally defined by the residual map $\textrm{res}_{2,3} : X^{[2,3]} \lra X$ and the second projection $\pi_{2} : X^{[2,3]} \lra X^{[3]}$. 
  It is known that $(\textrm{res}_{2,3}, \pi_2)$ is a small birational morphism. (cf. \cite{Gottsche}) Since $\mathcal{Z}_3 \lra X^{[3]}$ is finite flat, $\mathcal{Z}_3$ is normal and Cohen-Macaulay. Note that canonical sheaves of $X^{[2,3]}$ and $\mathcal{Z}_3$ are reflexive. By Kempf's criterion(cf. \cite{KKMSD.73}), $\mathcal{Z}_3$ has rational singularities.

One may also study the flatness and singularities of the first projection map  $ \pi_1 : X^{[k,k+1]} \lra X^{[k]}$. However, the fiber of $\pi_1$ can be reducible for $k \geq 3$, and its dimension is not constant for $k \geq 4$ when $\dim X = 2$. A similar phenomenon occurs in higher-dimensional cases, implying that $\pi_1$ cannot be flat for $k \geq 4$. Nonetheless, flatness does hold for $k = 3$.

\begin{prop} \label{main flat}
Let 
$X$ be a smooth projective variety. The first projection map $ \pi_1 : X^{[3,4]} \lra X^{[3]}$ is flat. 
\end{prop}

This proposition provides a basic example of the deformation of a smooth variety into a scheme with two components. By examining the singularities of the fiber of $\pi_1$, we establish the following theorem.

\begin{thm} \label{main}
 Let $X$ be a smooth projective variety. $X^{[3,4]}$ has canonical Gorenstein singularities. 
\end{thm}

    The structure of this paper is organized as follows. 
    In section 2, we introduce basic notions and constructions of nested Hilbert schemes. It is necessary to prove the flatness of the morphism $\pi_1$ and to verify that each fiber of $\pi_1$ is Gorenstein. We achieve this by computing the defining equations of the fibers based on the defining equation of non-reduced subscheme of $X$. 
    
   In section 3, We give a direct computation of the Hilbert polynomials of the fibers $\pi_1^{-1}([\eta])$ for the flatness of $\pi_1$. Before the proof, we need to construct an embedding of $X^{[3,4]}$ into $X^{[3]} \times \mathbb{P}^N$ for some $N$. This approach provides extrinsic understanding of the nested Hilbert schemes.

 In section 4, we verify the rational singularities of $X^{[3,4]}$ by examining those of $X^{[1,3,4]}$ via  \cite{Boutot.87}. Since $X^{[1,3,4]}$ is not easy to handle directly, we instead study $X^{[2,3,4]}$. We construct a two-step blow-up to resolve the singularities of $X^{[2,3,4]}$, producing  variety with potentially milder singularities. Furthermore, we provide partial generalizations of the results in \cite{Ryan.Taylor.22}, regarding on the singularities of $X^{[1,2,3]}, X^{[2,3,4]}, X^{[1,3,4]}$ and $X^{[1,2,3,4]}$.

 In section 5, we suggest open questions concerning the singularities of $X^{[4]}, \mathcal{Z}_4, X^{[2,4]}$ and $X^{[1,2,4]}$ motivated by the results in \cite{Katz.94}, \cite{Song.16} and \cite{Ram.Sam.24}.

\noindent \textbf{Acknowledgments.} This research was partially supported by the Institute for Basic Science (IBS-R032-D1) and the author would like to thank to Prof.Yongnam Lee for his guidance. I would like to thank to Prof.Joachim Jelisiejew for the reference \cite{Katz.94}. I also would like to thank to Mr.Erik Nikolov in Leibniz University for the reference \cite{Kleiman.90} and \cite{Lehn.04}.

\section{Computation of the fibers of $\pi_1$} \label{section 2}
\subsection{Structure of the two-step nested Hilbert scheme}
 Throughout this section, we assume $X$ is a smooth complex projective variety of dimension $n$, unless  otherwise stated. The two-step nested Hilbert scheme $X^{[k,k+1]}$ is equipped with two flat families $\mathcal{V} \subset \mathcal{W} \lra X^{[k,k+1]}$ of degree $k$ and $k+1$, respectively, whose fibers are closed subschemes of $X$. Since we regard $X^{[k,k+1]}$ as a closed subscheme of the product $X^{[k]} \times X^{[k+1]}$, there exist two natural projections
 $$\pi_{1,k,k+1}: X^{[k,k+1]} \lra X^{[k]} \; \textrm{and} \;\pi_{2,k,k+1}: X^{[k,k+1]} \lra X^{[k+1]}.$$ When the value of $k$ is clear, we typically simplify these projections by omitting $k$ and $k+1$, denoting them as $\pi_1$ and $\pi_2$. Here, we have the identification
 $$ \mathcal{V} \cong \pi_{1,X}^{-1}(\mathcal{Z}_k) \quad \textrm{and} \quad \mathcal{W} \cong \pi_{2,X}^{-1}(\mathcal{Z}_{k+1}),$$
where $\pi_{1,X}$ and $\pi_{2,X}$ are the natural projection morphism associated with the universal families $\mathcal{Z}_{k}$ and $\mathcal{Z}_{k+1}$, respectively.

 Denote the residual morphism $\textrm{res}_{k,k+1} : X^{[k,k+1]} \lra X$ by sending $([\eta],[\xi])$ to the support of the colon ideal $(I_{\xi} : I_{\eta})$(i.e. to the residual point of $\xi$ with respect to $\eta$). Let $\Gamma \subset X \times X^{[k,k+1]}$ denote the graph of $\textrm{res}_{k,k+1}$. Set theoretically, one has  $\mathcal{W} = \mathcal{V} \cup \Gamma$. We often omit the index and just denote by $\textrm{res}$ for residual morphisms.

  From now, we identify $X^{[k,k+1]}$ with the projective spectrum as follows:

  \begin{thm} (\cite[Theorem 2.8]{Kleiman.90}) \label{Kleiman}
      Let $X$ be a projective variety. Then, two-step nested Hilbert scheme $X^{[k,k+1]}$ is isomorphic with 
      $$\mathbb{P}(\mathcal{I}_{\mathcal{Z}_k}) := \textbf{Proj} \; \textrm{Sym}^{\bullet} \mathcal{I}_{\mathcal{Z}_k/X \times X^{[k]}}.$$
  \end{thm}

 \cite{Kleiman.90} and \cite{Ellingsrud.Stromme} utilized the two-step nested Hilbert scheme $X^{[k,k+1]}$ as an inductive tool for the cycle-theoretic computation of the Hilbert scheme of points. In this context, the two-step nested Hilbert scheme is often referred to as the incidence correspondence due to its role in relating subschemes of different lengths. \cite{Lehn.04} addressed and filled a gap of Fogarty's proof concerning the connectedness of the Hilbert scheme of point $X^{[k]}$.

 The projective spectrum construction reflects the equational properties of the two-step nested Hilbert schemes and proves useful when $X$ is a smooth complex surface. The universal family $\mathcal{Z}_k \subset X \times X^{[k]}$ has codimension 2, enabling the application of the Hilbert-Burch theorem, as shown in \cite{Ellingsrud.Stromme}, \cite{Song.16}, \cite{Ryan.Taylor.22} and \cite{Ram.Sam.24}. Furthermore, it is known that $X^{[k,k+1]}$ is isomorphic to the blow-up $\textrm{Bl}_{\mathcal{Z}_k} X \times X^{[k]}$ in this case.

\subsection{Minimal number of defining equations for the fibers of $\pi_1$} \label{section nummingens}

 On the other hand, consider the fiber of the morphism $\mathbb{P}(\mathcal{I}_{\mathcal{Z}_k}) \lra X \times X^{[k]}$. If the minimal number of generators of $\mathcal{I}_{\eta/X} \otimes k(p)$ exceeds $n+1$, the dimension of a projective space $\mathbb{P}(\mathcal{I}_{\eta/X} \otimes k(p))$ is greater than $n$, and consequently $\pi_{1}$ is no longer flat. For example, consider the case when $X$ is a smooth surface and $k$ is larger than $4$. There exists a punctual subscheme $\eta$ of length $k$ whose minimal number of local generator exceeds $3$. Then $\pi_1 : X^{[k,k+1]} \lra X^{[k]}$ is no longer flat. However, if the minimal number of local generators of $\mathcal{I}_{\eta/X} \otimes k(p)$ is exactly $3$, it may still be possible for $\pi_{1}$ to remain flat at such points of $X^{[k]}$, especially when $k=3$. We start from this observation.

\subsection{\'{E}tale local morphism to affine space}

  We show that the fibers of $\pi_1$ is Gorenstein and then show that $\pi_1$ is flat. The defining equations of the fiber of $\pi_1$ is obtained from the equations that define $\eta$ in $X$. For a direct computation, we refer \cite[Lemma 2.6]{Ram.Sam.24} for the followings. We identify a neighborhood $U \subset X$ of $\eta$ with a disjoint union of affine spaces $\mathbb{A}^n$ through an \'{e}tale morphism $ \varphi : U \lra \mathbb{A}^n \sqcup \cdots \sqcup \mathbb{A}^n=:A$. (It should be note that $U$ and $\varphi(U)$ may not necessarily be connected.) We can adjust $U$ and $\varphi$ so that, for any finite-length scheme $\eta'$ contained in $U$, the composition $\eta' \subset U \lra A$ remains a closed immersion. Consequently, we can apply \cite[Lemma 4.4]{Beh.Fan.08} with sufficient projective completion. By the argument in \cite[Lemma 4.4]{Ram.Sam.24}, we can identify $A^{[k]}$ \'{e}tale locally with a product  $ \Pi_{i}(\mathbb{A}^{n})^{[k_i]}$, where $\eta$ is the sum of punctual schemes of length $k_i$. 
 Thus, we may assume that $\eta$ is a punctual scheme and set $A = \mathbb{A}^n$.

There exist induced morphisms $V \lra (\mathbb{A}^n)^{[k]}$ and $V' \lra (\mathbb{A}^n)^{[k+1]}$for some open neighborhoods $V$ and $V'$ of $[\eta] \in X^{[k]}$ and $[\xi] \in X^{[k+1]}$, respectively. By a similar argument, we also have an induced morphism $W \lra (\mathbb{A}^n)^{[k,k+1]}$ for an open neighborhood $W$ of $([\eta], [\xi]) \in \pi_1^{-1}(V) \cap \pi_2^{-1}(V') \subset  X^{[k,k+1]}$.

 Since the compositions $\eta \subset X \lra \mathbb{A}^n$ and $\xi \subset X \lra \mathbb{A}^n$ are closed immersion, the residual point of the pair $([\varphi(\eta)], [\varphi(\xi)])$ can be identified with the residual point of the pair $([\eta], [\xi])$ via $\varphi : U \lra \mathbb{A}^n$. In particular, we obtain the following commutative diagrams :

 \begin{equation} \label{comm.diag.with affine space}
\xymatrix@=30pt{
 W \subset X^{[k,k+1]} \ar[r]^-{\pi_1} \ar[d]& V \subset X^{[k]} \ar[d] &  W \subset X^{[k,k+1]} \ar[r]^-{\textrm{res}_{k,k+1}} \ar[d] & U \subset X \ar[d]^{\varphi} \\
(\mathbb{A}^n)^{[k,k+1]} \ar[r]^-{\pi_1} & (\mathbb{A}^n)^{[k]} & (\mathbb{A}^n)^{[k,k+1]} \ar[r]^-{\textrm{res}_{k,k+1}} & \mathbb{A}^n
   }
\end{equation}

\subsection{Direct computation of the fibers of $\pi_1$}

 The main strategy is to regard $X$ as an affine space $\mathbb{A}^n$ near $\eta$ and directly compute the defining equations of the blow-ups. From now, we set $k = 3$.

\begin{prop} \label{fiber lci}
 The fiber $\pi_1^{-1}([\eta])$ is a Gorenstein scheme.
\end{prop}
\begin{proof} 

 Note that there are three kinds of length 3 subscheme of $X$ :
 \begin{enumerate}
     \item [(a)] distinct three points
     \item [(b)] curvi-linear subschemes
     \item [(c)] non-curvi-linear subschemes
 \end{enumerate}

 Since $X^{[3]}$ is smooth at $[\eta]$ and the projection morphism $\mathcal{Z}_3 \lra X^{[3]}$ is flat, both inclusions $[\eta] \in X^{[3]}$ and $\eta \times [\eta] \hookrightarrow \mathcal{Z}_3$ are regular immersions. Thus, $\textrm{Tor}_1^R(R/I, R/J)$ vanishes where $R$ is the local ring $ \mathcal{O}_{X \times X^{[3]}, (x, [\eta])}$ and $I,J$ are ideals of $R$ defining $\mathcal{Z}_3$ and $X  \times [\eta]$, respectively. Hence, the restriction of the ideal $\mathcal{I}_{\mathcal{Z}_3}|_{X \times [\eta]}$ is isomorphic to the ideal $\mathcal{I}_{\eta/X}$. Then, the fiber $\pi_1^{-1}([\eta])$ is isomorphic to the projective spectrum $\textbf{Proj} \;\textrm{Sym}^{\bullet} \mathcal{I}_{\eta/X}$.

   Let's consider the case (a) and (b) first. In these cases, $\xi$ can be embedded into some affine plane. By the standard argument in deformation theory, $X^{[3,4]}$ is smooth at point $([\eta],[\xi])$. By the subsection \S\ref{section nummingens}, dimension of the fiber $\pi_1^{-1}([\eta])$ is $n$ in these cases. By Hironaka's lemma \cite[Theorem 23.1]{Matsumura.89}, $\pi_1$ is flat at this point and hence the fiber is Gorenstein. (Actually, it is l.c.i.)

  We only need to examine the fiber of $\pi_1$ for the case (c). We may assume that $\eta$ is supported on the origin of $\mathbb{A}^n$ with the following defining equations : 
\begin{equation*} \label{eqn of non-curvi-linear scheme}
 x_1^2, x_1x_2, x_2^2, x_3, \cdots, x_n.
 \end{equation*}   Assume $n \geq 3$. The fiber $\pi_1^{-1}([\eta])$ is a closed subvariety of $\mathbb{A}^n_{x_1,\cdots, x_n} \times \mathbb{P}^n_{u_0,\cdots, u_n}$ defined by following  equations for $i,j \geq 3$ 
 \begin{equation} \label{equation}
     x_1u_1-x_2u_0, x_1u_2-x_2u_1, x_1^2u_i-x_iu_0, x_1x_2u_i-x_iu_1, x_2^2u_i-x_iu_2, x_iu_j - x_ju_i.
 \end{equation}
Let us focus on the locus where $u_n \neq 0$. On this locus, the equations reduced to:  
\begin{equation*} \label{eqn for blow up of non-lci}
    x_1u_1-x_2, x_1u_2-x_2u_1, x_1^2-x_nu_0, x_1x_2-x_nu_1, x_2^2-x_nu_2, x_1^2u_i-x_i \;(\textnormal{for} \; i \geq 3).
\end{equation*}
 Thus, the locus where $u_n \neq 0$ is isomorphic to the product of $\mathbb{A}^{n-3}$ with the spectrum of $$ A := \mathbb{C}[x_1,x_2,x_n, u_0, u_1,u_2]/(x_1u_1-x_2u_0, x_1u_2-x_2u_1, x_1^2-x_nu_0, x_1x_2-x_nu_1, x_2^2-x_nu_2).$$
  The ring $A$ is Gorenstein by Macaulay2. (We can check this when the base field is the field of rational numbers. Since the base field extension by $\mathbb{C}$ is faithfully flat, it does not harm the Gorenstein property.)

 By symmetry, the same arguments holds for the locus $u_3 \neq 0, \cdots, u_{n-1} \neq 0$. Now consider the cases where $u_0 \neq 0, u_1 \neq 0$ and $u_2 \neq 0$. We only need to analyze the cases two cases $u_0 \neq 0$ and $u_1 \neq 0$ due to symmetry.

 On the locus $u_0 \neq 0$, we simplify the equations (\ref{equation}) to:
 \begin{equation*} \label{u_0 nonzero}
     x_1u_1-x_2, x_1u_2-x_2u_1, x_1^2u_i-x_i \;(\textnormal{for} \; i \geq 3).
 \end{equation*}
The locus satisfying $u_0 \neq 0$ is isomorphic to product of $\mathbb{A}^{n-2}$ with the spectrum of the ring 
$$ A := \mathbb{C}[x_1,x_2,u_1,u_2]/(x_1u_1-x_2, x_1u_2-x_2u_1).$$
 This ring is Gorenstein and hence the fiber $\pi_1^{-1}([\eta])$ is Gorenstein. (If $n=2$, $\Spec A$ is an open subscheme of the fiber $\pi_1^{-1}([\eta])$.)

 Similarly, on the locus $u_1 \neq 0$, we simplify equations (\ref{equation}) to:
 \begin{equation*} \label{u_1 nonzero}
     x_1-x_2u_0, x_1u_2-x_2, x_1x_2u_i-x_i
 \end{equation*}
The locus satisfying $u_1 \neq 0$ is isomorphic to product of $\mathbb{A}^{n-2}$ with 
$$ \Spec \; \mathbb{C}[x_1,x_2,u_0,u_2]/(x_1-x_2u_0, x_1u_2-x_2).$$
 As in the above argument, the locus $u_1 \neq 0$ of the fiber of $\pi_1$ is also Gorenstein.

Thus, we conclude that the fiber of $\pi_1$ is always a Gorenstein scheme.
\end{proof}

\begin{rem} \label{remark on case (b)}
 We can check case (b) by direct computation. We may assume that $\eta$ is supported on the origin of $\mathbb{A}^n$ with the following defining equations : 
\begin{equation*} \label{eqn of non-curvi-linear scheme}
 x_1^{\ell},x_2, x_3, \cdots, x_n,
 \end{equation*} where $\ell $ is either $2$ or $3$. 
The fiber $\pi_1^{-1}([\eta]) \subset \mathbb{A}^n_{x_1,\cdots, x_n} \times \mathbb{P}^{n-1}_{u_1 , \cdots, u_n}$ is a closed subvariety defined by following equations for $i,j \geq 2$
 \begin{equation} \label{equation_curvi_fiber}
     x_1^{\ell}u_j - x_ju_1, x_iu_j-x_ju_i.
 \end{equation}

   We examine these equations on the locus $u_1 \neq 0$. In this case, the equations (\ref{equation_curvi_fiber}) reduced to:
 \begin{equation*}
     x_1^{\ell}u_j-x_j, x_iu_j-x_ju_i.
 \end{equation*}
 This locus is nonsingular, as verified by computing the rank of the Jacobian matrix of the equations.

  Next, we consider the locus $u_j \neq 0$ for $j \geq 2$. Without loss of generality, we may set $j = 2$. Then, the locus $u_2 \neq 0$ is defined by the equations for $i, k \geq 3$  
  \begin{equation*}
      x_1^{\ell}-x_2u_1, x_1^{\ell}u_i-x_iu_1, x_i-x_2u_i, x_iu_k-x_ku_i.
  \end{equation*}

The locus where $u_2 \neq 0$ is isomorphic to the product of $\mathbb{A}^{n-2}$ with the spectrum 
$$ \Spec \; \mathbb{C}[x_1, x_2, u_1]/(x_1^{\ell} -x_2u_1).$$
This locus is Gorenstein and hence, $\pi_1^{-1}([\eta])$ is Gorenstein for the case (b).

\end{rem}

 Let $\eta$ be a length-$3$ subscheme of $X$ that is \'{e}tale locally defined by the equations given in  (\ref{eqn of non-curvi-linear scheme}). We observe that the fiber $\pi_1^{-1}([\eta])$ indeed has two components of the same dimension. In the component containing the locus where $x_n \neq 0$, the following expression hold:
 $$ u_0 = \frac{x_1^2}{x_n}, u_1 = \frac{x_1x_2 }{x_n}, u_2 = \frac{x_2^2}{x_n}, $$
and on this locus, the relation $u_0u_2-u_1^2$ is satisfied. This component is thus isomorphic to an open subvariety of blow up of $X$. We refer to it as the blow-up component and denote it by $B_n$. The blow-up component parameterizes those pairs $([\eta], [\xi])$ for which $\xi$ is a flat limit of length-$4$ scheme that splits as the disjoint union of $\eta$ and a single reduced point. The component supported on the locus $x_n = 0$ is isomorphic to $\mathbb{P}^n$. We refer to it as the projective space component and denote it by $P_n$. 

\begin{proof}[Proof 1 of proposition \ref{main flat}] 
 To prove the flatness of $\pi_1 : X^{[3,4]} \lra X^{[3]}$, it is enough to consider when $X = \mathbb{P}^n$.  
 We have shown that $\pi_1$ is flat over the locus of $X^{[3]}$ corresponding to cases (a) and (b) in proposition \ref{fiber lci}. We now extend flatness by showing constancy of the Hilbert polynomial of the fiber of $\pi_1$. Let $\eta$ be a length 3 subscheme of case (c). Choose a suitable coordinates and embedding  $\mathbb{P}^2_{x_0, x_1, x_2} \hookrightarrow \mathbb{P}^n_{x_0, x_1, \cdots, x_n}$ such that $\eta \subset \mathbb{P}^2_{x_0,x_1,x_2}$. We temporally denote the first projection $X^{[3,4]} \lra X^{[3]}$  by $g_2$ and $g_n$ for $X = \mathbb{P}^2$ and $X = \mathbb{P}^n$, respectively. This induces the  commutative diagram:
\begin{equation*} 
\xymatrix@=30pt{
   (\mathbb{P}^2)^{[3,4]}  \ar[r]^-{g_2} \ar[d] & (\mathbb{P}^2)^{[3]} \ar[d] \\
  (\mathbb{P}^n)^{[3,4]}  \ar[r]^-{g_n} & (\mathbb{P}^n)^{[3]}, }
\end{equation*}

 where the left and right vertical arrow are closed immersions. Because $[\eta]$ lies in the closure of the smooth locus of $(\mathbb{P}^2)^{[3]}$, we can choose such a curve $C \subset (\mathbb{P}^2)^{[3]}$ passing through $[\eta]$ such that $C - [\eta]$ lies in the open locus consisting of distinct $3$ points. The restriction $g_2^{-1}(C) \lra C$ is flat. In particular,  
  $$\overline{g_2^{-1}(C - [\eta])} = g_2^{-1}(C).$$
  We also have that $$g_n^{-1}(C-[\eta]) \lra C - [\eta]$$ is flat. Let
  $$g' : \overline{g_n^{-1}(C - [\eta])} \to C$$
 be the closure of that family over $C$ which is indeed flat over the whole curve $C$ by \cite[III, Proposition 9.8]{Hartshorne.77}.

  It remains to check the case (c) at the special fiber over $[\eta]$. Recall that $g_2^{-1}([\eta]) = B_2 \cup P_2$ and  $g_n^{-1}([\eta]) = B_n \cup P_n$ hold. By the computation in the proof of proposition \ref{fiber lci} we have the inclusions $B_2 \subset B_n$ and $P_2 \subset P_n$, while $B_2 \not\subset P_n$ and $P_2 \not\subset B_n$.

  By construction, $g'$ is flat, so its fiber is pure of dimension $n$. Since $(g')^{-1}([\eta])$  contains the plane-fiber $g_2^{-1}([\eta])$, it must contain the entire $g_n^{-1}([\eta])$. Therefore the closure $\overline{g_n^{-1}(C - [\eta])}$ agrees with $g_n^{-1}(C)$, and hence the restriction $g_n|_{g_n^{-1}(C)} : g_n^{-1}(C) \lra C$ is flat. This shows that the Hilbert polynomial of the fibers $\pi_1^{-1}([\eta])$ is constant along $[\eta]$ showing the flatness of $\pi_1$.
\end{proof}

Therefore, $\pi_1$ is a flat Gorenstein morphism, which implies that $\mathbb{P}(\mathcal{I}_{\mathcal{Z}_3})$ is Gorenstein. In particular, the nested Hilbert scheme of points $X^{[3,4]}$ is normal and isomorphic to the blow-up $\textrm{Bl}_{\mathcal{Z}_3}(X \times X^{[3]})$.

\section{Alternate proof of flatness of $\pi_1$}
We present a canonical projective embedding of $X^{[k,k+1]}$ and compute the Hilbert polynomial of the fiber $\pi_1^{-1}([\eta])$ when $k=3$, assuming proposition \ref{fiber lci}. This computation directly shows proposition \ref{main flat}. To construct a natural ample line bundle on $X^{[k,k+1]}$, we begin with the following definition.

 \begin{defn} [cf. {\cite{BS88} and \cite{catanese90}}]
    A line bundle $\mathcal{L}$ on a complete algebraic variety $X$ over an algebraically closed field $k$ is \textit{d-very ample} if, for every zero-dimensional subscheme $Z$ of $Y$ with length less than or equal to $d+1$, the restriction map $$r_Z: H^0(Y, \mathcal{L}) \to H^0(Y, \mathcal{L} \otimes \mathcal{O}_Z)$$ is surjective.
\end{defn}

 For example, $0$-very ampleness is equivalent to global generation and $1$-very ampleness is equivalent to the original very ampleness. Tensor product of $d_1$-very ample line bundle and $d_2$-very ample line bundle is $d_1+d_2$-very ample, allowing us to construct higher very ample bundles.

 On the other hand, consider two projections $p_1$ and $p_2$ from $X \times X^{[k]}$ to each factors. Let $L$ be a line bundle on $X$. Since $p_2$ is finite flat of degree $k$, the sheaf $p_{2,*}(p_1^*L \otimes \mathcal{O}_{\mathcal{Z}_k})$ is locally free of rank $k$. This sheaf, known as the tautological bundle, is denote by $E_{k,L}$. If $L$ is $k$-very ample, there is a natural morphism 
 $$H^0(X, L) \otimes \mathcal{O}_{X^{[k]}} \longrightarrow E_{k,L}$$
 that is surjective. According  to \cite{catanese90}, the determinant bundle $N_{k,L}:= \textrm{det} E_{k,L}$ is very ample on $X^{[k]}$. The linear system $|N_{k, L}|$ defines an embedding $X^{[k]} \hookrightarrow \mathbb{P}^m$ via the Pl\"{u}cker embedding.

  Now assume that $L$ is $k+1$-very ample. The linear system $|\pi_2^*N_{k+1, L}|$ defines a morphism $$X^{[k,k+1]} \subset X^{[k]} \times X^{[k+1]} \longrightarrow X^{[k]} \times \mathbb{P}^{\ell},$$ for some $\ell$ 
  implying that $\pi_2^*N_{k+1, L}$ is relatively ample over $X^{[k]}$.

    Let $\mathcal{I}_{k,k+1}$ denote the kernel of the morphism $\pi_{2,X}^*\mathcal{O}_{\mathcal{Z}_{k+1}} \lra \pi_{1,X}^*\mathcal{O}_{\mathcal{Z}_{k}}$. By \cite[Proposition 2.2]{Kleiman.90}, the sheaf $\mathcal{I}_{k,k+1}$ is flat over $X^{[k,k+1]}$, and it coincides with the tautological line bundle $\mathcal{O}_{\mathbb{P}(\mathcal{I}_{\mathcal{Z}_k})}(1)$ on the projective scheme $\mathbb{P}(\mathcal{I}_{\mathcal{Z}_{k}})$. By tensoring with $\operatorname{pr}_1^*L$ and then taking the direct image under $\operatorname{pr}_2$ yields the short exact sequence
    $$ 0 \lra \textrm{res}^*L \otimes \mathcal{I}_{k,k+1} \lra \pi_2^*E_{k+1,L} \lra \pi_1^*E_{k,L} \lra 0,$$
  showing that $\pi_2^*N_{k+1,L} = \pi_1^*N_{k,L} \otimes \textrm{res}^*L \otimes \mathcal{I}_{k,k+1},$ so that 
 $$\textrm{res}^*L \otimes \mathcal{I}_{k,k+1}$$
 is relatively ample over $X^{[k]}$. This justifies the computation of the Hilbert polynomial of the family. 

\begin{proof}[Proof 2 of proposition \ref{main flat}]
 We already showed in the proposition \ref{fiber lci} that $\pi_1$ is flat at the point $([\eta],[\xi])$ where $[\eta]$ is either distinct points or curvilinear subscheme. We show the flatness of $\pi_1$ by comparing the Hilbert polynomial of the fibers $\pi_1^{-1}([\eta])$ at distinct points and at non-curvilinear schemes. Consider the case when $\eta$ consists of three distinct points of $X$. The fiber $\pi_1^{-1}([\eta])$ is isomorphic to the blow-up of $X$ along those three points. Let $\mu : X' \longrightarrow X$ denote the blow-up, with corresponding exceptional divisors $E_1, E_2$ and $E_3$. The Hilbert polynomial of $\pi_1^{-1}([\eta])$ is given by: 
 \begin{equation*}
     h^0(X', \mu^*L^{\ell} \otimes \mathcal{O}(-\ell E_1-\ell E_2 - \ell E_3)).
 \end{equation*}

This formula is equivalent to the following:
\begin{equation*}
    h^0(X, L^{\ell}) - 3\sum_{i=0}^{\ell - 1}h^0(E_1, \mathcal{O}_{E_1}(iE_1))
\end{equation*}
 since $\mu$ is a birational morphism onto a normal variety $X$. Note that $h^0(E_1, \mathcal{O}_{E_1}(iE_1))$ is the same as $h^0(\mathbb{P}^{n-1}, \mathcal{O}(i))$. Therefore, the Hilbert polynomial of the fiber $\pi_1^{-1}([\eta])$ is given by 
 \begin{equation} \label{hilb poly distinct pts}
  h^0(X, L^{\ell}) - 3\sum_{i=0}^{\ell - 1}h^0(\mathbb{P}^{n-1}, \mathcal{O}(i)). 
 \end{equation}

  Next, we consider the fiber at a non-curvilinear scheme. 
Recall that we denote $B_n$ and $P_n$ by the blow-up component and the exceptional projective space component of the fiber $\pi_1^{-1}([\eta])$. Denote by $E_{\eta}$ the exceptional divisor of the blow-up $\mu : B_n \longrightarrow X$ along $\eta$.  Then, $E_{\eta} \subset P_n$ can be regarded as a (possibly singular) quadric $Q$. The Hilbert polynomial of the fiber $\pi_1^{-1}([\eta])$ is provided as:
 \begin{equation} \label{hilb poly noncurcilinear}
     h^0(B_n, \mu^*L^{\ell} \otimes \mathcal{O}(-\ell E_{\eta})) + h^0(\mathbb{P}^n, \mathcal{O}(\ell)) - h^0(Q, \mathcal{O}_Q(\ell)).     
 \end{equation}

  By Macaulay2, we compute the depth of blow-up component $B_n$ and check that it is normal. Thus the first term can be rewritten as:
 $$ h^0(X, L^{\ell}) - \sum_{i=0}^{\ell - 1}h^0(E_{\eta}, \mathcal{O}_{E_{\eta}}(iE_{\eta})). $$

 Let $I = (x_3, \cdots, x_n) $ and $m = (x_1, x_2)$ be ideals of polynomial ring $k[x_1, \cdots, x_n]$. Let $J = I + m^2$. The dimension $h^0(E_{\eta}, \mathcal{O}_{E_{\eta}}(iE_{\eta}))$ equals the length of module $J^i/J^{i+1}$. Since $J^i = I^i + m^2I^{i-1} + \cdots + m^{2i-2}I + m^{2i}$, we need only to count the minimal generators of $J^i$. This number is given by:
 \begin{align*}
 &3\left( {n+i-3 \choose i} + 3 {n+i-4 \choose i-1} + \cdots + (2i-3){n-1 \choose 2} + (2i-1){n-2 \choose 1} + (2i+1) \right)\\
 - &\left( 2 {n+i-4 \choose i-1} + \cdots + (2i-4){n-1 \choose 2} + (2i-2){n-2 \choose 1} + 2i \right).   
 \end{align*}

 This simplifies to:
 $$\sum_{i=0}^{\ell - 1}h^0(E_{\eta}, \mathcal{O}_{E_{\eta}}(iE_{\eta}))=  3{n+\ell - 1 \choose n} + {n + \ell-2 \choose n}. $$

On the other hand, $h^0(Q, \mathcal{O}_Q(\ell))$ can be computed using the short exact sequence:
$$ 0 \longrightarrow \mathcal{O}_{\mathbb{P}^n}(\ell - 2) \longrightarrow \mathcal{O}_{\mathbb{P}^n}(\ell) \longrightarrow \mathcal{O}_Q(\ell) \longrightarrow 0. $$

 For $\ell \gg 2,$ we have: $$h^0(\mathbb{P}^n, \mathcal{O}(\ell)) - h^0(Q, \mathcal{O}_Q(\ell)) = {n+\ell-2 \choose \ell - 2}. $$

  Combining these results, we see that expressions (\ref{hilb poly distinct pts}) and (\ref{hilb poly noncurcilinear}) coincide when $\ell$ is large. Thus, $\pi_1 : X^{[3,4]} \longrightarrow X^{[3]}$ is flat. 
 \end{proof}

\section{Showing canonical singularities via blow-ups} \label{section cano sing via blow up}
 We show that $X^{[3,4]} \cong \textrm{Bl}_{\mathcal{Z}_3}(X \times X^{[3]})$ has rational singularities. By comparing induced universal families one checks
$$X^{[1,3,4]} \cong X^{[1,3]} \times_{X^{[3]}} X^{[3,4]} \quad \textrm{and} \quad X^{[2,3,4]} \cong X^{[2,3]} \times_{X^{[3]}} X^{[3,4]}.$$ Since the projection $\mathcal{Z}_3 \lra X^{[3]}$ is faithfully flat, the morphism $X^{[1,3,4]} \lra X^{[3,4]}$ is also faithfully flat. If we show that $X^{[1,3,4]}$ has rational singularities, then $X^{[3,4]}$ does as well by \cite{Boutot.87}. Due to some technical issues, we resolve the singularities by $X^{[2,3,4]}$ by applying the morphism $(\textrm{res}_{2,3}, \pi_2) : X^{[2,3]} \lra \mathcal{Z}_3$. As in \cite[section 2.2]{Ryan.Taylor.22}, the following holds: 
$$X^{[2,3,4]} \cong \textrm{Bl}_{\mathcal{W}} (X \times X^{[2,3]}).$$

Note that the schemes  $\mathcal{W}$, $\mathcal{V}$ and $\Gamma$ are reduced since they are finite and flat over $X^{[2,3]}$. Therefore $\mathcal{W} = \mathcal{V} \cup \Gamma$.   
 Since $\mathcal{W}$ is reducible, it is inconvenient to study its blow-up directly. Instead, we consider the two-step blow-up $\textrm{Bl}_{\widetilde{\Gamma}}\textrm{Bl}_{\mathcal{V}}(X\times X^{[2,3]})$ where $\widetilde{\Gamma}$ denotes the strict transform of $\Gamma$ under the first blow-up $\textrm{Bl}_{\mathcal{V}}(X \times X^{[2,3]}) \lra X \times X^{[2,3]}$.

  \begin{prop} \label{rational singularities of blow-up}
      The two step blow-up $\textrm{Bl}_{\widetilde{\Gamma}}\textrm{Bl}_{\mathcal{V}}(X \times X^{[2,3]})$ has rational singularities.
  \end{prop} 
\begin{proof} 
We keep the notation of the preceding discussion. First note that, because $X^{[2,3]} \lra X^{[2]}$ is flat, the blow-up can be identified with the self-fiber product $$\textrm{Bl}_{\mathcal{V}}(X \times X^{[2,3]}) \cong X^{[2,3]} \times_{X^{[2]}} X^{[2,3]}.$$ Let $F_2$ denote the exceptional divisor of the blow-up $X^{[2,3]} \lra X \times X^{[2]}$. One checks that $\Gamma \cap \mathcal{V} \cong F_2$; hence the strict transform $\widetilde{\Gamma}$ of $\Gamma$ is isomorphic to the blow-up $\textrm{Bl}_{F_2} X^{[2,3]}$. Since $F_2$ is a Cartier divisor of $X^{[2,3]}$, the blow-up is isomorphic to $\Gamma$. Under the identification 
$$\textrm{Bl}_{\mathcal{V}} (X \times X^{[2,3]}) \cong X^{[2,3]}\times_{X^{[2]}}X^{[2,3]},$$ 
the subscheme $\widetilde{\Gamma}$ is the diagonal copy of $X^{[2,3]}$. Consider the projection $$p : \textrm{Bl}_{\widetilde{\Gamma}}\textrm{Bl}_{\mathcal{V}}(X \times X^{[2,3]}) \lra X^{[2,3]},$$ given by composing the blow-ups with the second projection $X^{[2,3]}\times_{X^{[2]}}X^{[2,3]}$. As the Tor group argument in the proof of proposition \ref{fiber lci}, for a closed point $\zeta = ([\eta], [\xi])$ the fiber $p^{-1}(\zeta)$ is isomorphic to the blow-up of the fiber $\pi_1^{-1}([\eta])$ at the point corresponding to $\zeta$; in symbols
$$p^{-1}(\zeta) \cong \textrm{Bl}_{\zeta}\pi_1^{-1}([\eta]).$$ 
\begin{itemize}
    \item  If $\eta$ consists of two distinct points, then $\textrm{Bl}_{\zeta}\pi_1^{-1}([\eta])$ is smooth, so its blow-up at any point is smooth; hence the corresponding fiber $p^{-1}(\zeta)$ is smooth. 
    \item Recall remark \ref{remark on case (b)}. If $\eta$ is a length 2 fat point(the non-reduced case), then (locally near the singular locus) $\pi_1^{-1}([\eta])$ is analytically  isomorphic to the affine quadric hypersurface $$V(f) \subset \mathbb{A}^n, \quad f = x_1^2-x_2x_3,$$ and the singular point corresponds to the origin $o \in \mathbb{A}^n$. The blow up $\textrm{Bl}_{o}V(f)$ is isomorphic to the total space of $\mathcal{O}(-1)$ over the projective quadric $Q_{n-1}$; in particular $\textrm{Bl}_{o}V(f)$ has at worst canonical Gorenstein singularities (because $Q_{n-1}$ has canonical Gorenstein singularities).
\end{itemize}
Hence every fiber $p^{-1}(\zeta)$ is normal with dimension $n$ and hence $p$ is a flat morphism by \cite[III, Theorem 9.11]{Hartshorne.77}. Since the fiber  $p^{-1}(\zeta)$ has rational singularities, by \cite{Elkik.78}, $\textrm{Bl}_{\widetilde{\Gamma}}\textrm{Bl}_{\mathcal{V}}(X \times X^{[2,3]})$ also has rational singularities.  

\end{proof}

The following lemma helps to compare two-step blow-up with single-step blow-up that we desire to investigate.

 \begin{lem} \label{comparison between blow-ups}
 Let the notation be as above. Then there exists a morphism $$\textrm{Bl}_{\widetilde{\Gamma}}\textrm{Bl}_{\mathcal{V}}(X \times X^{[2,3]}) \lra \textrm{Bl}_{\mathcal{W}}(X\times X^{[2,3]}),$$ induced by the universal property of blowing up. Moreover, it is a small birational morphism(no divisor is contracted). 
 \end{lem}
 \begin{proof}
 Write $B :=\textrm{Bl}_{\widetilde{\Gamma}}\textrm{Bl}_{\mathcal{V}}(X \times X^{[2,3]}) $. By \cite[III, Theorem 7.17]{Hartshorne.77}, $B$ itself is the blow-up of $X \times X^{[2,3]}$ along some closed subscheme $Y \subset X \times X^{[2,3]}$. By construction, one has $$Y - \Gamma \cap \mathcal{V} \cong \mathcal{W} - \Gamma \cap \mathcal{V},$$ so the underlying support of $Y$ equals the support of $\mathcal{W}$. Let $b : \textrm{Bl}_{Y}(X \times X^{[2,3]}) \lra X \times X^{[2,3]}$ denote the blow-up morphism and wirte $I_Y$ and $I_{\mathcal{W}}$ for the ideal sheaves of $Y$ and $\mathcal{W}$ in $X \times X^{[2,3]}$. 
      On $B$ we have the inclusion of inverse images of ideals
      $$b^{-1}I_{Y}\cdot\mathcal{O}_{B} \subset b^{-1}I_{\mathcal{W}}\cdot\mathcal{O}_{B},$$ and by the agreement of supports their radicals coincide.

       Let $s$ be a nilpotent local section of inverse image ideal $b^{-1}I_Y\cdot \mathcal{O}_B$. Since the embedding $Y - \Gamma \cap \mathcal{V} \subset X \times X^{[2,3]}$ is a regular embedding of codimension $n$, the restriction of $s$ on the open subscheme $B - b^{-1}(\Gamma \cap \mathcal{V})$ is zero. As $B:=\textrm{Bl}_Y (X \times X^{[2,3]})$ is normal and the codimension of $b^{-1}(\Gamma \cap \mathcal{V}) \subset B$ is at least 2, $s$ is the zero section. Hence, the exceptional divisor of $b$ is reduced. So the ideal $b^{-1}I_{Y}\cdot\mathcal{O}_{B}$ is radical; hence the two inverse image ideals agree: $b^{-1}I_{Y}\cdot\mathcal{O}_{B} = b^{-1}I_{\mathcal{W}}\cdot\mathcal{O}_{B}$. Thus $b^{-1}I_{\mathcal{W}}\cdot\mathcal{O}_{B}$ is a line bundle on $B$. By the universal property of the blow-up, there exists a morphism $B \lra \textrm{Bl}_{\mathcal{W}}(X\times X^{[2,3]})$. Along the proof, it is shown that this morphism is small birational.   
 \end{proof}

  \begin{proof}[Proof of theorem \ref{main}]
   By the previous argument and lemma \ref{comparison between blow-ups}, we obtain a morphism 
   $$ \textrm{Bl}_{\widetilde{\Gamma}}\textrm{Bl}_{\mathcal{V}}(X \times X^{[2,3]})\lra X^{[1,3,4]},$$
   which is a small birational morphism. The source is a variety with rational singularities by proposition \ref{rational singularities of blow-up}, while the target is normal and Cohen-Macaulay. As in the  introduction(where we treated $\mathcal{Z}_3$), the target therefore has rational singularities. Consequently $X^{[3,4]}$ has rational singularities by \cite{Boutot.87}. Being Gorenstein, it follows that its singularities are canonical Gorenstein singularities by \cite{Elkik.81}.
  \end{proof}

 Finally, analogous statements holds for other nested Hilbert schemes; these follows by essentially the same arguments with only minor modifications.

\begin{cor} \label{cor}
    $X^{[1,2,3]}, X^{[2,3,4]}$ and $X^{[1,2,3,4]}$ have canonical Gorenstein singularities. $X^{[1,3,4]}$ has rational singularities.
\end{cor}
\begin{proof}
 By base change, $X^{[1,2,3]} \lra X^{[1,2]}$ and $X^{[2,3,4]} \longrightarrow X^{[2,3]}$ are flat Gorenstein morphisms with fibers having rational singularities. Thus, $X^{[1,2,3]}$ and $X^{[2,3,4]}$ are Gorenstein varieties and hence have canonical singularities by the similar arguments.  $X^{[1,2,3,4]}$ also has canonical Gorenstein singularities by similar arguments.

 We already showed the rational singularities of $X^{[1,3,4]}$ during the proof of the main theorem.
\end{proof}

\section{Open question}

 We complete this paper with two questions. $(\mathbb{P}^3)^{[4]}$ is normal and Cohen-Macaulay by \cite{Katz.94}. Since $X^{[4]}$ is isomorphic with $(\mathbb{P}^3)^{[4]}$ in its \'{e}tale local chart when $X$ is a smooth $3$-fold, $X^{[4]}$ is also normal and Cohen-Macaulay. This leads to the following question.      

\begin{Q}
 Is $X^{[4]}$ normal and Cohen-Macaulay regardless of the dimension of $X$? Does it have rational singularities?
\end{Q}

If $\dim X > 2$, we can not apply the Hilbert-Burch theorem for the free resolution of $\mathcal{O}_{\mathcal{Z}_k}$(cf.\cite{Ellingsrud.Stromme} or \cite{Song.16}). Despite this limitation, it appears that $X^{[3]} ,X^{[4]}, \mathcal{Z}_3$ and $\mathcal{Z}_4$ have similar geometric properties in higher dimensional cases as seen in the case of surfaces. The following question comes naturally. 

\begin{Q}
    Does $\mathcal{Z}_4$ have rational but non $\mathbb{Q}$-Gorenstein singularities?
\end{Q}


\begin{thebibliography}{EGHP05}

\bibitem[BF08]{Beh.Fan.08}
Kai Behrend and Barbara Fantechi, \emph{Symmetric obstruction theories and Hilbert schemes of points on threefolds}, Algebra Number Theory  \textbf{2} (2008), 313--345.

\bibitem[BS88]{BS88}
Mauro C. Beltrametti and Andrew J. Sommese, \emph{Zero cycles and kth order embeddings of 
smooth projective surfaces}, Problems in the theory of surfaces and their classification(Cortona, 1988) (1991), 33--48.

\bibitem[Bou87]{Boutot.87}
     Jean-Fran\c cois Boutot, \emph{Singularit\'es rationnelles et quotients par les groupes
              r\'eductifs}, Invent. Math. \textbf{88} (1987), 65--68.
              
\bibitem[CG90]{catanese90}
Fabrizio Catanese and Lothar G{\"o}ttsche, \emph{d-very-ample line bundles and embeddings 
of Hilbert schemes of 0-cycles}, Manuscripta math. \textbf{68} (1990), no.1, 
337--341.

\bibitem[Che98]{Cheah.98}
Jan Cheah, \textit{Cellular decompositions for nested {H}ilbert schemes of points}, Pacific. J. Math. \textbf{183} (1998), 39--90.

\bibitem[Elk78]{Elkik.78}
Ren\'{e}e Elkik, \emph{Singularit\'{e}s rationnelles et d\'{e}formations}, Invent. \!Math. \textbf{47} (1978), 139--147.
 
\bibitem[Elk81]{Elkik.81}
Ren\'{e}e Elkik, \emph{Rationalit\'{e} des singularit\'es canoniques}, Invent. Math. \textbf{64} (1981), no.1, 1--6.

\bibitem[ES98]{Ellingsrud.Stromme}
Geir Ellingsrud and Stein Arild Str\o{}mme, \emph{An intersection number for the punctual {H}ilbert scheme of a surface}, Trans. Amer. Math. Soc. \textbf{350} (1998), 2547--2552.

\bibitem[For73]{Fogarty.73}
 John Fogarty, \emph{Algebraic families on an algebraic surface. {II}. {T}he 
{P}icard scheme of the punctual {H}ilbert scheme}, Amer. J. Math. \textbf{95} (1973), 660--687.

 \bibitem[G\"{o}93]{Gottsche}             
 Lothar G\"{o}ttsche, \emph{Hilbertschemata nulldimensionaler {U}nterschemata glatter {V}ariet\"{a}ten}, Bonner Mathematische Schriften [Bonn Mathematical Publications] \textbf{243}, Dissertation, Rheinische Friedrich-Wilhelms-Universit\"at Bonn, Universit\"at Bonn, Mathematisches Institut, Bonn (1993), x+254.

\bibitem[Har77]{Hartshorne.77}
Robin Hartshorne, \emph{Algebraic Geometry}, Graduate Texts in Math., vol. 52, Springer-Verlag, New York, 1977.

\bibitem[Katz94]{Katz.94}
Sheldon H. Katz, \emph{The desingularization of {${\textrm{Hilb}}^4{\mathbf{P}}^3$} and its {B}etti numbers}, Zero-dimensional schemes ({R}avello, 1992), de Gruyter, Berlin, (1994), 231--242.
              
\bibitem[KKMSD73]{KKMSD.73}
George R. Kempf, Finn Faye Knudsen, David Mumford, and Bernard Saint-Donat, \emph{Toroidal embeddings. I}, Lecture Notes in Mathematics, Vol. 339, Springer-Verlag, Berlin-New York, (1973), viii+209.

\bibitem[Kl90]{Kleiman.90}
Steven L. Kleiman, \emph{Multiple-point formulas. {II}. {T}he {H}ilbert scheme}, In: Xambó-Descamps, S. (ed.) Enumerative Geometry, pp. 101–138. Springer, Berlin (1990).

\bibitem[Lehn04]{Lehn.04}
Manfred Lehn, \emph{Lectures on Hilbert schemes}, CRM proceedings and lecture notes, Vol.38, American Mathematical Society, (2004), 1--30.

\bibitem[Mat89]{Matsumura.89}
Hideyuki Matsumura, \emph{Commutative ring theory}, Cambridge Studies in Advanced Mathematics, Vol.8, Second ed., Translated from the Japanese by M. Reid, Cambridge University Press, Cambridge, 1989.

\bibitem[RS24]{Ram.Sam.24}
Ritvik Ramkumar and Alessio Sammartano, \emph{Rational singularities of nested {H}ilbert schemes}, Int. Math. Res. Not. IMRN \textbf{2} (2024), 1061--1122.

\bibitem[RT22]{Ryan.Taylor.22}
Tim Ryan and Gregory Taylor, \textit{Irreducibility and singularities of some nested {H}ilbert schemes}, J. Algebra \textbf{609} (2022), 380--406.

\bibitem[Song16]{Song.16}
Lei Song, \emph{On the universal family of {H}ilbert schemes of points on a surface}, J. Algebra \textbf{456} (2016), 348--354.
            
	
\end{thebibliography}
\end{document}